\documentclass[11pt]{article}
\usepackage{amssymb}
\newtheorem{precor}{{\bf Corollary}}

\newenvironment{cor}{\begin{precor}{\hspace{-0.5
               em}{\bf.\ }}}{\end{precor}}
\newtheorem{precon}{{\bf Conjecture}}

\newenvironment{con}{\begin{precon}{\hspace{-0.5
               em}{\bf.\ }}}{\end{precon}}
\newtheorem{prealphcon}{{\bf Conjecture}}

\newtheorem{predefin}{{\bf Definition}}

\newenvironment{defin}[1]{\begin{predefin}{\hspace{-0.5
                   em}{\bf.\ }}{\rm #1}\hfill{$\spadesuit$}}{\end{predefin}}
\newtheorem{preexm}{{\bf Example}}

\newtheorem{preappl}{{\bf Application}}

\newtheorem{prelem}{{\bf Lemma}}

\newtheorem{preproof}{{\bf Proof.\ }}

\newenvironment{proof}[1]{\begin{preproof}{\rm
               #1}\hfill{$\blacksquare$}}{\end{preproof}}
\newtheorem{pretheorem}{{\bf Theorem}}

\newenvironment{theorem}{\begin{pretheorem}{\hspace{-0.5
               em}{\bf.\ }}}{\end{pretheorem}}
\newtheorem{prealphtheorem}{{\bf Theorem}}

\newtheorem{prealphlem}{{\bf Lemma}}

\newtheorem{prepro}{{\bf Proposition}}

\newenvironment{pro}{\begin{prepro}{\hspace{-0.5
               em}{\bf.\ }}}{\end{prepro}}
\newtheorem{preprb}{{\bf Problem}}

\newtheorem{prerem}{{\bf Remark}}

\newtheorem{preapp}{{\bf Application}}

\newtheorem{prequ}{{\bf Question}}

\def\conct[#1,#2]{\mbox {${#1} \leftrightarrow {#2}$}}
\def\dconct[#1,#2]{\mbox {${#1} \rightarrow {#2}$}}
\def\deg[#1,#2]{\mbox {$d_{_{#1}}(#2)$}}
\def\mindeg[#1]{\mbox {$\delta_{_{#1}}$}}
\def\maxdeg[#1]{\mbox {$\Delta_{_{#1}}$}}
\def\outdeg[#1,#2]{\mbox {$d_{_{#1}}^{^+}(#2)$}}
\def\minoutdeg[#1]{\mbox {$\delta_{_{#1}}^{^+}$}}
\def\maxoutdeg[#1]{\mbox {$\Delta_{_{#1}}^{^+}$}}
\def\indeg[#1,#2]{\mbox {$d_{_{#1}}^{^-}(#2)$}}
\def\minindeg[#1]{\mbox {$\delta_{_{#1}}^{^-}$}}
\def\maxindeg[#1]{\mbox {$\Delta_{_{#1}}^{^-}$}}

\def\dre[#1,#2,#3]{\mbox {${\cal E}^{^{#3}}(#1,#2)$}}
\def\var[#1,#2]{\mbox {${\rm Var}_{_{#1}}(#2)$}}
\def\ls[#1]{\mbox {$\xi^{^{#1}}$}}
\def\hom[#1,#2]{\mbox {${\rm Hom}({#1},{#2})$}}
\def\onvhom[#1,#2]{\mbox {${\rm Hom^{v}}(#1,#2)$}}
\def\onehom[#1,#2]{\mbox {${\rm Hom^{e}}(#1,#2)$}}
\def\core[#1]{\mbox {$#1^{^{\bullet}}$}}
\def\cay[#1,#2]{\mbox {${\rm Cay}({#1},{#2})$}}
\def\sch[#1,#2,#3]{\mbox {${\rm Sch}({#1},{#2},{#3})$}}
\def\cays[#1,#2]{\mbox {${\rm Cay_{s}}({#1},{#2})$}}
\def\dirc[#1]{\mbox {$\stackrel{\rightarrow}{C}_{_{#1}}$}}
\def\cycl[#1]{\mbox {${\bf Z}_{_{#1}}$}}

\begin{document}

\begin{center}
{\Large \bf On Fall Colorings of Graphs}\\
\vspace{0.3 cm}
{\bf Saeed Shaebani}\\
{\it Department of Mathematical Sciences}\\
{\it Institute for Advanced Studies in Basic Sciences (IASBS)}\\
{\it P.O. Box {\rm 45195-1159}, Zanjan, Iran}\\
{\tt s\_shaebani@iasbs.ac.ir}\\ \ \\
\end{center}
\begin{abstract}
\noindent A fall $k$-coloring of a graph $G$ is a proper
$k$-coloring of $G$ such that each vertex of $G$ sees all $k$
colors on its closed neighborhood. We denote ${\rm Fall}(G)$ the
set of all positive integers $k$ for which $G$ has a fall
$k$-coloring. In this paper, we study fall colorings of
lexicographic product of graphs and categorical product of graphs
and answer a question of \cite{dun} about fall colorings of
categorical product of complete graphs. Then, we study fall
colorings of union of graphs. Then, we prove that fall
$k$-colorings of a graph can be reduced into proper $k$-colorings
of graphs in a specified set. Then, we characterize fall
colorings of Mycielskian of graphs. Finally, we prove that for
each bipartite graph $G$, ${\rm Fall}(G^{c})\subseteq \ \{\ \chi
(G^{c})\ \}$ and it is polynomial time to decision  whether or
not ${\rm Fall}(G^{c})=\{\ \chi (G^{c})\ \}$ .

\noindent {\bf Keywords:}\ {  fall Coloring, lexicographic
product, categorical product.}\\
{\bf Subject classification: 05C}
\end{abstract}
\section{Introduction}
All graphs considered in this paper are finite and simple
(undirected, loopless and without multiple edges). Let $G=(V,E)$
be a graph and $k\in \mathbb{N}$ and $[k]:=\{i|\ i\in
\mathbb{N},\ 1\leq i\leq k \}$. A $k$-coloring (proper k-coloring)
of $G$ is a function $f:V\rightarrow [k]$ such that for each
$1\leq i\leq k$, $f^{-1}(i)$ is an independent set. We say that
$G$ is $k$-colorable whenever $G$ admits a $k$-coloring $f$, in
this case, we denote $f^{-1}(i)$ by $V_{i}$ and call each $1\leq
i\leq k$, a color (of $f$) and each $V_{i}$, a color class (of
$f$). The minimum integer $k$ for which $G$ has a $k$-coloring,
is called the chromatic number of G and is denoted by $\chi(G)$.

Let $G$ be a graph, $f$ be a $k$-coloring of $G$ and $v$ be a
vertex of $G$. The vertex $v$ is called colorful ( or
color-dominating or $b$-dominating) if each color $1\leq i\leq k$
appears on the closed neighborhood of $v$ (\ $f(N[v])=[k]$ ). The
$k$-coloring $f$ is said to be a fall $k$-coloring (of $G$) if
each vertex of $G$ is colorful. There are graphs $G$ for which
$G$ has no fall $k$-coloring for any positive integer $k$. For
example, $C_{5}$ ( a cycle with 5 vertices) and graphs with at
least one edge and one isolated vertex, have not any fall
$k$-colorings for any positive integer $k$. The notation ${\rm
Fall}(G)$ stands for the set of all positive integers $k$ for
which $G$ has a fall $k$-coloring. Whenever ${\rm
Fall}(G)\neq\emptyset$, we call $\min({\rm Fall}(G))$ and
$\max({\rm Fall}(G))$, fall chromatic number of $G$ and fall
achromatic number of $G$ and denote them by $\chi_{f}(G)$ and
$\psi_{f}(G)$, respectively. The terminology fall coloring was
firstly introduced in 2000 in \cite{dun} and has received
attention recently, see
\cite{MR2096633},\cite{MR2193924},\cite{dun},\cite{las}.

\section{Fall colorings of lexicographic product of graphs}

Let $G$ and $H$ be graphs. The lexicographic product of G and H
is defined the graph with vertex set $V(G)\times V(H)$ and edge
set $\{\ \{(x_{1},y_{1}),(x_{2},y_{2})\}\ |\ x_{1},x_{2}\in V(G)\
{\rm and}\ y_{1},y_{2}\in V(H)\ {\rm and}\ [\ (\{x_{1},x_{2}\}\in
E(G))\ {\rm or}\ (x_{1}=x_{2}, \{y_{1},y_{2}\}\in E(H))\ ]\ \}$.
For each $x\in V(G)$, the induced subgraph of $G[H]$ on
$\{x\}\times V(H)$ is denoted by $H_{x}$.

Note that $G[H]$ and $H[G]$ are not necessarily isomorphic. For
example, let $G:=K_{2}$ and $H$ be the complement of $G$. $G[H]$
has 4 edges and $H[G]$ has 2 edges and therefore, they are not
isomorphic. But lexicographic product of graphs is associative up
to isomorphism ( For arbitrary graphs $G_{1}$, $G_{2}$ and
$G_{3}$, $(G_{1}[G_{2}])[G_{3}]$ and $G_{1}[G_{2}[G_{3}]]$ are
isomorphic.).

\begin{theorem}{
\label{Fall} Let $G$ and $H$ be graphs and $k\in {\rm Fall}(G[H])$
and $f$ be a fall $k$-coloring of $G[H]$. Then, for each $x\in
V(G)$, $S_{x}:=f(V(H_{x}))$ forms a fall $|S_{x}|$-coloring of
$H_{x}$.}
\end{theorem}
\begin{proof}{ Let $x\in V(G)$ and $(x,y)$ be an arbitrary vertex of
 $H_{x}$
and its color be $\alpha$. Then, for each $\beta\in
S_{x}\setminus\{\alpha\}$, there exists a vertex $(a,b)$ of $G[H]$
adjacent with $(x,y)$ which is colored $\beta$. Obviously $a=x$,
otherwise, since $\beta\in S_{x}$, there exists a vertex
$(x,z)\in V(H_{x})$ colored $\beta$. $(x,y)$ is adjacent with
$(a,b)$ and $x\neq a$, so $\{x,a\}\in E(G)$ and therefore, $(x,z)$
and $(a,b)$ are adjacent in $G[H]$ and both of them are colored
$\beta$, which is a contradiction. Therefore, $a=x$ and $(a,b)\in
V(H_{x})$. Hence, $S_{x}$ forms a fall $|S_{x}|$-coloring of
$H_{x}$. }\end{proof}

\begin{cor}{ Let $G$ and $H$ be graphs. Then,
${\rm Fall}(G[H])\neq\emptyset\Rightarrow {\rm
Fall}(H)\neq\emptyset$, or equivalently, ${\rm
Fall}(H)=\emptyset\Rightarrow {\rm Fall}(G[H])=\emptyset$.}
\end{cor}
\begin{cor}{ Let $G$ and $H$ be graphs such that
${\rm Fall}(G[H])\neq\emptyset$. Then, ${\rm
Fall}(H)\neq\emptyset$ and for each fall $k$-coloring $f$ of
$G[H]$ and each $x\in V(G)$, $\chi_{f}(H)\leq
|f(V(H_{x}))|\leq\psi_{f}(H)$.}
\end{cor}

There are pairs of graphs $(G,H)$ for which ${\rm
Fall}(G)=\emptyset$ but ${\rm Fall}(G[H])\neq\emptyset$. For
example, ${\rm Fall}(C_{5})=\emptyset$ but $C_{5}[K_{2}]$ has a
fall 5-coloring. First let's label the vertices of $C_{5}[K_{2}]$
lexicographically: $1:=(1,1),\ 2:=(1,2),\ 3:=(2,1), \ldots,\
10:=(5,2)$. Here is a fall 5-coloring $f$ of $C_{5}[K_{2}]$:
$f(1)=1,\ f(2)=2,\ f(3)=3,\ f(4)=4,\ f(5)=1,\ f(6)=5,\ f(7)=2,\
f(8)=4,\ f(9)=5,\ f(10)=3$. Also, there are pairs of graphs
$(G,H)$ for which ${\rm Fall}(G)=\emptyset$ and ${\rm
Fall}(H)\neq\emptyset$ and ${\rm Fall}(G[H])=\emptyset$. For
example, ${\rm Fall}(C_{5})=\emptyset$ and ${\rm
Fall}(K_{1})\neq\emptyset$ and ${\rm Fall}(C_{5}[K_{1}])={\rm
Fall}(C_{5})=\emptyset$. The next theorem shows that if ${\rm
Fall}(G)\neq\emptyset$ and ${\rm Fall}(H)\neq\emptyset$, then,
${\rm Fall}(G[H])\neq\emptyset$.

\begin{theorem}{
\label{subset} Let $G$ and $H$ be graphs for which ${\rm
Fall}(G)\neq\emptyset$ and ${\rm Fall}(H)\neq\emptyset$. Then,
$\{\ \sum_{i=1}^{s}k_{i}\ |\ s\in {\rm Fall}(G),\  \forall 1\leq
i\leq s:\  k_{i}\in {\rm Fall}(H)\ \}\subseteq {\rm Fall}(G[H])$.}
\end{theorem}
\begin{proof}{ Let $s\in {\rm Fall}(G)$ and $g:V(G)\rightarrow[s]$ be a
fall $s$-coloring of $G$ and for each $1\leq i\leq s$, $k_{i}\in
{\rm Fall}(H)$ and $h_{i}$ be a fall $k_{i}$-coloring of $H$.
Let's color each vertex $(x,y)$ of $G[H]$ by color
$(g(x),h_{g(x)}(y))$. Indeed, let's consider the function
$f:V(G[H])\rightarrow S:=\{\ (g(x),h_{g(x)}(y))\ |\ (x,y)\in
V(G)\times V(H)\ \}$ which assigns to each $(x,y)$ of $G[H]$,
$(g(x),h_{g(x)}(y))$. For each adjacent vertices $(x,y)$ and
$(a,b)$ in $G[H]$, $\{x,a\}\in E(G)$ or $(x=a$ and $\{y,b\}\in
E(H))$. So, $g(x)\neq g(a)$ or $(g(x)=g(a)$ and $h_{g(x)}(y)\neq
h_{g(a)}(b))$. Therefore, $(g(x),h_{g(x)}(y))\neq
(g(a),h_{g(a)}(b))$. This shows that $f$ is a
$(\sum_{i=1}^{s}k_{i})$-coloring of $G[H]$ such that uses exactly
$\sum_{i=1}^{s}k_{i}$ colors. Now let's show that $f$ is a fall
$(\sum_{i=1}^{s}k_{i})$-coloring of $G[H]$. For each $(x,y)\in
V(G[H])$ and each $(\alpha,\beta)\in S\setminus\{\
(g(x),h_{g(x)}(y))\ \}$, there is a vertex $(u,v)$ of $G[H]$
colored $(\alpha,\beta)$, or equivalently,
$(g(u),h_{g(u)}(v))=(\alpha,\beta)$. Now, there are two cases:

Case I) The case that $g(x)=g(u)$. In this case,
$h_{g(x)}=h_{g(u)}$ and $h_{g(x)}(y)\neq h_{g(u)}(v)$. Since
$h_{g(x)}$ is a fall $k_{g(x)}$-coloring of $H$, there exists a
vertex $z\in V(H)$ such that $\{z,y\}\in E(H)$ and
$h_{g(x)}(z)=h_{g(u)}(v)$. The vertex $(x,z)$ of $G[H]$ is
adjacent with $(x,y)$ and its color is
$f((x,z))=(g(x),h_{g(x)}(z))=(g(u),h_{g(u)}(v))=(\alpha,\beta)$.

Case II) The case that $g(x)\neq g(u)$. Since $g$ is a fall
$s$-coloring of $G$, there exists a vertex $z\in V(G)$ such that
$\{x,z\}\in E(G)$ and $g(z)=g(u)$. So, $h_{g(u)}(v)=h_{g(z)}(v)$.
The vertex $(z,v)$ is adjacent with $(x,y)$ in $G[H]$ and
$f((z,v))=(g(z),h_{g(z)}(v))=(g(u),h_{g(u)}(v))=(\alpha,\beta)$.

Hence, $f$ is a fall $(\sum_{i=1}^{s}k_{i})$-coloring of $G[H]$.
Therefore, $\{\ \sum_{i=1}^{s}k_{i}\ |\ s\in {\rm Fall}(G)$,\
$\forall 1\leq i\leq s$:\  $k_{i}\in {\rm Fall}(H)\ \}\subseteq
{\rm Fall}(G[H])$. }\end{proof}

\begin{cor}{ Let $G$ and $H$ be graphs for which
${\rm Fall}(G)\neq\emptyset$ and ${\rm Fall}(H)\neq\emptyset$.
Then, $\chi_{f}(G[H])\leq \chi_{f}(G)\chi_{f}(H)\leq
\psi_{f}(G)\psi_{f}(H)\leq \psi_{f}(G[H])$.}
\end{cor}

$\chi_{f}(G[H])$ and $\chi_{f}(G)\chi_{f}(H)$ are not necessarily
equal. For example, $\chi_{f}(C_{9})=3$ and $\chi_{f}(K_{2})=2$.
Therefore, $\chi_{f}(C_{9})\chi_{f}(K_{2})=6$, but
$\chi_{f}(C_{9}[K_{2}])\leq5$, first let's label the vertices of
$C_{9}[K_{2}]$ lexicographically: 1:=(1,1), 2:=(1,2), 3:=(2,1),
...,\ 18:=(9,2). Here is a fall 5-coloring $f$ of $C_{9}[K_{2}]$:
$f(1)=1,\ f(2)=4,\ f(3)=2,\ f(4)=3,\ f(5)=5,\ f(6)=1,\ f(7)=4,\
f(8)=2,\ f(9)=3,\ f(10)=1,\ f(11)=5,\ f(12)=2,\ f(13)=4,\ f(14)=3,
f(15)=1,\ f(16)=2,\ f(17)=5,\ f(18)=3$. Also,
$\psi_{f}(G)\psi_{f}(H)$ and $\psi_{f}(G[H])$ are not necessarily
equal. For example, $\psi_{f}(C_{8})=2$ and $\psi_{f}(K_{2})=2$
and therefore, $\psi_{f}(C_{8})\psi_{f}(K_{2})=4$. But
$\psi_{f}(C_{8}[K_{2}])\geq5$. First let's label the vertices of
$C_{8}[K_{2}]$ lexicographically: $1:=(1,1),\ 2:=(1,2),\
3:=(2,1), \ldots,\ 16:=(8,2)$. Here is a fall 5-coloring $f$ of
$C_{8}[K_{2}]$: $f(1)=1,\ f(2)=2,\ f(3)=3,\ f(4)=4,\ f(5)=5,\
f(6)=1,\ f(7)=2,\ f(8)=3,\ f(9)=4,\ f(10)=1,\ f(11)=5,\ f(12)=2,\
f(13)=3,\ f(14)=1,\ f(15)=5,\ f(16)=4$.

Theorem \ref{subset} says that if $G$ and $H$ are graphs for
which ${\rm Fall}(G)\neq\emptyset$ and ${\rm
Fall}(H)\neq\emptyset$, Then, $\{\ \sum_{i=1}^{s}k_{i}\ |\ s\in
{\rm Fall}(G),\  \forall 1\leq i\leq s:\  k_{i}\in {\rm Fall}(H)\
\}\subseteq {\rm Fall}(G[H])$. Since $5\in {\rm
Fall}(C_{9}[K_{2}])$ and $5\notin \{\ \sum_{i=1}^{s}k_{i}\ |\ \
s\in {\rm Fall}(C_{9}),\ \forall\ 1\leq i\leq s$:\  $k_{i}\in {\rm
Fall}(K_{2})\ \},\ {\rm Fall}(G[H])$ and $\{\ \sum_{i=1}^{s}k_{i}\
|\ s\in {\rm Fall}(G),\ \forall\ 1\leq i\leq s:\  k_{i}\in {\rm
Fall}(H)\}$ are not necessarily equal in this theorem.

\begin{theorem}{
There are pairs of graphs $(G,H)$ for which ${\rm
Fall}(G)\neq\emptyset$ and ${\rm Fall}(H)\neq\emptyset$ and the
following strictly inequality holds.

$\chi_{f}(G[H])<\chi_{f}(G)\chi_{f}(H)<\psi_{f}(G)\psi_{f}(H)<\psi_{f}(G[H])$.}
\end{theorem}
\begin{proof}{ Set $G:=C_{6}\bigvee C_{8}\bigvee C_{9}$ ( the join of $C_{6}$ and $C_{8}$ and $C_{9}$) and
$H:=K_{2}$. Since $(C_{6}\bigvee C_{8}\bigvee C_{9})[K_{2}]$ and
$(\ C_{6}[K_{2}]\ )\bigvee (\ C_{8}[K_{2}]\ )\bigvee (\
C_{9}[K_{2}]\ )$ are isomorphic, $\chi_{f}((C_{6}\bigvee
C_{8}\bigvee
C_{9})[K_{2}])=\chi_{f}(C_{6}[K_{2}])+\chi_{f}(C_{8}[K_{2}])+\chi_{f}(C_{9}[K_{2}])\leq4+4+5=13$
and $\psi_{f}((C_{6}\bigvee C_{8}\bigvee
C_{9})[K_{2}])=\psi_{f}(C_{6}[K_{2}])+\psi_{f}(C_{8}[K_{2}])+\psi_{f}(C_{9}[K_{2}])\geq6+5+6=17$
. Also, $\chi_{f}(C_{6}\bigvee C_{8}\bigvee C_{9})=7$ and
$\psi_{f}(C_{6}\bigvee C_{8}\bigvee C_{9})=8$ and
$\chi_{f}(K_{2})=\psi_{f}(K_{2})=2$, as desired. }\end{proof}

\begin{theorem}{
For each $\varepsilon>0$, There exists a pair of graphs $(S,T)$
for which
$\min\{\psi_{f}(S[T])-\psi_{f}(S)\psi_{f}(T),\psi_{f}(S)\psi_{f}(T)-\chi_{f}(S)\chi_{f}(T),\chi_{f}(S)\chi_{f}(T)-\chi_{f}(S[T])\}\geq\varepsilon$.
}\end{theorem}
\begin{proof} {With no loss of generality, we can assume that
 $\varepsilon$
is a natural number. Set $G:=C_{6}\bigvee C_{8}\bigvee C_{9}$ and
$S:=K_{\varepsilon}[G]$ and $T:=K_{2}$. Since $S[T]$ and
$K_{\varepsilon}[G[T]]$ are isomorphic and
$\chi_{f}(K_{\varepsilon}[G[T]])=\varepsilon\chi_{f}(G[T])$ and
$\psi_{f}(K_{\varepsilon}[G[T]])=\varepsilon\psi_{f}(G[T])$, the
theorem implies.}
\end{proof}

One can easily observe that if $G$ and $H$ are graphs such that
${\rm Fall}(G[H])\neq\emptyset$, then,
$\chi_{f}(G[H])\geq\omega(G)\chi_{f}(H)$. The next clear
proposition introduces a sufficient condition for equality.

\begin{pro}{
Let $G$ and $H$ be graphs such that ${\rm Fall}(G)\neq\emptyset$
and ${\rm Fall}(H)\neq\emptyset$ and $\chi_{f}(G)=\omega(G)$.
Then,
 $\chi_{f}(G[H])=\chi_{f}(G)\chi_{f}(H)=\omega(G)\chi_{f}(H)$.}
\end{pro}
\begin{cor}{
If $G$ is a tree or a complete graph or $C_{2k}$(for some
$k\in\mathbb{N}\setminus\{1\}$) and $H$ is a graph such that ${\rm
Fall}(H)\neq\emptyset$, then,
$\chi_{f}(G[H])=\chi_{f}(G)\chi_{f}(H)=\omega(G)\chi_{f}(H)$.}
\end{cor}

Corollary \ref{Fall} says that in every fall $k$-coloring of
$G[H]$ and each $x\in V(G)$, the number of colors appear on
$V(H_{x})$ is at most $\psi_{f}(H)$. Hence, $\psi_{f}(G[H])\leq
(\delta(G)+1)\psi_{f}(H)$. The following clear proposition
introduces a condition for equality.

\begin{pro}{
Let $G$ and $H$ be graphs for which ${\rm Fall}(G)\neq\emptyset$
and ${\rm Fall}(H)\neq\emptyset$ and $\psi_{f}(G)=\delta(G)+1$.
Then,
$\psi_{f}(G[H])=\psi_{f}(G)\psi_{f}(H)=(\delta(G)+1)\psi_{f}(H)$.}
\end{pro}
\begin{cor}{
 If $G$ is a tree or a complete graph or $C_{3k}$(for some
$k\in\mathbb{N}$) and $H$ is a graph such that ${\rm
Fall}(H)\neq\emptyset$, then,
$\psi_{f}(G[H])=\psi_{f}(G)\psi_{f}(H)=(\delta(G)+1)\psi_{f}(H)$.}
\end{cor}

\section{Type-II graph homomorphisms and lexicographic product of
graphs}

Now we study a type of graph homomorphisms that is related to
fall colorings of graphs.

\begin{defin}{ Let $G$ and $H$ be graphs. A function
$f:V(G)\rightarrow V(H)$ is called a type-II graph homomorphism
from $G$ to $H$ if $f$ satisfies the following two conditions.
\\
\\
1) $ \{u,v\}\in E(G)\Rightarrow \{f(u),f(v)\}\in E(H)$.
\\
2) $ \{u_{1},v_{1}\}\in E(H)\Rightarrow \forall v\in
f^{-1}(v_{1}): \exists u\in f^{-1}(u_{1})$ s.t $\{u,v\}\in E(G)$.
}\end{defin}

Type-II graph homomorphisms introduced by Laskar and Lyle in 2009
in \cite{las}.\ They showed that for any graph $G$, $k\in {\rm
Fall}(G)$ iff there exists a type-II graph homomorphism from $G$
to $K_{k}$. Note that every type-II graph homomorphism from a
graph $G$ to a complete graph, is surjective. If $f_{1}$ is a
type-II graph homomorphism from $G$ to $H$ and $f_{2}$ is a
type-II graph homomorphism from $H$ to $I$, then, $f_{2}of_{1}$ is
a type-II graph homomorphism from $G$ to $I$. Also, if there
exists a type-II graph homomorphism from $G$ to $H$ and $k\in {\rm
Fall}(H)$, then, $k\in {\rm Fall}(G)$. If there exists a type-II
graph homomorphism from $G_{1}$ to $G_{2}$ and a type-II graph
homomorphism from $H_{1}$ to $H_{2}$, then, there exists a
type-II graph homomorphism from $G_{1}\square H_{1}$ to
$G_{2}\square H_{2}$. We prove a similar theorem for
lexicographic product of graphs.

\begin{theorem} { Let $G_{1}$, $G_{2}$, $H_{1}$ and $H_{2}$ be
graphs and $f_{1}$ be a type-II graph homomorphism from $G_{1}$ to
$G_{2}$ and $f_{2}$ be a surjective type-II graph homomorphism
from $H_{1}$ to $H_{2}$. Then, there exists a type-II graph
homomorphism $f_{3}$ from $G_{1}[H_{1}]$ to $G_{2}[H_{2}]$.}
\end{theorem}

\begin{proof} {
Let $f_{3}:V(G_{1}[H_{1}])\rightarrow V(G_{2}[H_{2}])$ be defined
the function which assigns to each $(g,h)\in V(G_{1}[H_{1}])$,
$f_{3}((g,h))=(f_{1}(g),f_{2}(h))$. For each $\{\
(x_{1},y_{1}),(x_{2},y_{2})\ \}\in E(G_{1}[H_{1}])$,
$\{x_{1},x_{2}\}\in E(G_{1})$ or ($x_{1}=x_{2}$ and
$\{y_{1},y_{2}\}\in E(H_{1})$). Therefore, $\{\
f_{1}(x_{1}),f_{1}(x_{2})\ \}\in E(G_{2})$ or
($f_{1}(x_{1})=f_{1}(x_{2})$ and $\{\ f_{2}(y_{1}),f_{2}(y_{2})\
\}\in E(H_{2})$). Hence, $\{\
(f_{1}(x_{1}),f_{2}(y_{1})),(f_{1}(x_{2}),f_{2}(y_{2}))\ \}\in
E(G_{2}[H_{2}])$ and consequently, the property 1 holds. Now for
each $\{\ (\alpha_{1},\beta_{1}),(\alpha_{2},\beta_{2})\ \}\in
E(G_{2}[H_{2}])$ and each $(u_{1},v_{1})\in
f_{3}^{-1}((\alpha_{1},\beta_{1}))$, there are two cases:

Case I) The case that $\{\alpha_{1},\alpha_{2}\}\in E(G_{2})$.
Since $f_{1}$ is a type-II graph homomorphism and $u_{1}\in
f_{1}^{-1}(\alpha_{1})$, there exists $u_{2}\in
f_{1}^{-1}(\alpha_{2})$ such that $\{u_{1},u_{2}\}\in E(G_{1})$.
Surjectivity of $f_{2}$ implies that there exists $v_{2}\in
f_{2}^{-1}(\beta_{2})$. Therefore, $(u_{2},v_{2})\in
f_{3}^{-1}((\alpha_{2},\beta_{2}))$ and $\{\
(u_{1},v_{1}),(u_{2},v_{2})\ \}\in E(G_{1}[H_{1}])$ and
accordingly, the property 2 holds.

Case II) The case that $\alpha_{1}=\alpha_{2}$ and
$\{\beta_{1},\beta_{2}\}\in E(H_{2})$. In this case, $u_{1}\in
f_{1}^{-1}(\alpha_{2})$ and since $f_{2}$ is a type-II graph
homomorphism and $v_{1}\in f_{2}^{-1}(\beta_{1})$, there exists
$v_{2}\in f_{2}^{-1}(\beta_{2})$ such that $\{v_{1},v_{2}\}\in
E(H_{1})$. Hence, $(u_{1},v_{2})\in
f_{3}^{-1}((\alpha_{2},\beta_{2}))$ and $\{\
(u_{1},v_{1}),(u_{1},v_{2})\ \}\in E(G_{1}[H_{1}])$ and
therefore, the property 2 holds. Thus, $f_{3}$ is a type-II graph
homomorphism. }\end{proof}

\begin{cor}{ If $G$ and $H$ are graphs such that $r_{1}\in
{\rm Fall}(G)$ and $r_{2}\in {\rm Fall}(H)$, then

$\chi_{f}(G[H])\leq \chi_{f}(G[K_{r_{2}}])\leq
\chi_{f}(K_{r_{1}}[K_{r_{2}}])\leq\psi_{f}(K_{r_{1}}[K_{r_{2}}])\leq\psi_{f}(G[K_{r_{2}}])\leq\psi_{f}(G[H])$.}
\end{cor}

\section{Fall colorings of categorical product of graphs}

Let $G_{1}=(V_{1},E_{1})$ and $G_{2}=(V_{2},E_{2})$ be graphs. The
graph $G_{1}\times G_{2}:=(V_{1}\times V_{2},\{\ \{\
(x_{1},y_{1}),(x_{2},y_{2})\ \}\ |\ \{x_{1},x_{2}\}\in E(G_{1})\
{\rm and}\  \{y_{1},y_{2}\}\in E(G_{2})\ \})$ is called the
categorical product of $G$ and $H$.

Categorical product of graphs is commutative and associative up to
isomorphism (For each arbitrary graphs $G_{1},\ G_{2}$ and
$G_{3},\ G_{1}\times G_{2}$ and $G_{2}\times G_{1}$ are
isomorphic, also, $(G_{1}\times G_{2})\times G_{3}$ and
$G_{1}\times (G_{2}\times G_{3})$ are isomorphic.). For arbitrary
graphs $G$ and $H$, if $E(G)=\emptyset$ or $E(H)=\emptyset$, then,
$E(G\times H)=\emptyset$ and therefore, $G\times H$ has only a
fall 1-coloring and ${\rm Fall}(G\times H)=\{1\}$. Thus,
hereafter, let's focus on nonempty edge set graphs, unless stated
otherwise. Firstly, note that ${\rm Fall}(G:=(\ \{a,b,c,d\}\ ,\
\{\ \{a,b\},\{b,c\},\{c,a\},\{d,a\}\ \}\ )=\emptyset$ and ${\rm
Fall}(G\times G)=\emptyset$. Secondly, note that ${\rm
Fall}(C_{5}:=(\ \{0,1,2,3,4\}\ ,\ \{\
\{0,1\},\{1,2\},\{2,3\},\{3,4\},\{4,0\}\ \}\ ))=\emptyset$, but
the function $f:V(C_{5}\times C_{5})\rightarrow [5]$ which
assigns to each $(i,j)$ of $V(C_{5}\times C_{5})$,
$f((i,j)):=$(the arithmetic residue of $i+2j$ modulo 5)$+1$ where
the last $+$ is the natural summation in $\mathbb{Z}$, is a fall
5-coloring of $C_{5}\times C_{5}$, and therefore, ${\rm
Fall}(C_{5}\times C_{5})\neq\emptyset$. The next theorem implies
that if ${\rm Fall}(G)\neq\emptyset$ or ${\rm
Fall}(H)\neq\emptyset$, then, ${\rm Fall}(G\times
H)\neq\emptyset$.

\begin{theorem}{ \label{thm4}
For each $n\in\mathbb{N}$ and each arbitrary graphs
$G_{1},\ldots,G_{n}$,

$\forall 1\leq i\leq n: {\rm Fall}(G_{i})\subseteq {\rm
Fall}(\times_{i=1}^{n}G_{i})$.}
\end{theorem}
\begin{proof}{ Since categorical product of graphs is commutative and
associative up to isomorphism, it suffices to prove that ${\rm
Fall}(G_{1})\subseteq {\rm Fall}(G_{1}\times G_{2})$. If ${\rm
Fall}(G_{1})=\emptyset$, the theorem holds trivially. For each
$k\in {\rm Fall}(G_{1})$, there exists a fall $k$-coloring $f$ of
$G_{1}$. Now, the function $g:V(G_{1}\times G_{2})\rightarrow
[k]$ which assigns to each $(u,v)\in V(G_{1}\times G_{2})$,
$g((u,v))=f(u)$ is a fall $k$-coloring of $G_{1}\times G_{2}$ and
therefore, $k\in {\rm Fall}(G_{1}\times G_{2})$. Hence, ${\rm
Fall}(G_{1})\subseteq {\rm Fall}(G_{1}\times G_{2})$. }\end{proof}

\begin{cor}{\label{twoinequalities} For each $n\in \mathbb{N}$
and each arbitrary graphs $G_{1},\ldots,G_{n}$ such that for each
$i\in [n]$, ${\rm Fall}(G_{i})\neq\emptyset$, the following
inequalities hold.

$\chi_{f}(\times_{i=1}^{n}G_{i})\leq \min\{\ \chi_{f}(G_{i})\ |\
i\in
 [n]\ \}\leq \max\{\ \psi_{f}(G_{i})\ |\ i\in [n]\ \}\leq
         \psi_{f}(\times_{i=1}^{n}G_{i})$.}
\end{cor}

Now again type-II graph homomorphisms:

\begin{theorem}{  Let $G_{1}$, $G_{2}$, $H_{1}$ and $H_{2}$ be
graphs and $f_{1}$ be a type-II graph homomorphism from $G_{1}$ to
$G_{2}$ and $f_{2}$ be a type-II graph homomorphism from $H_{1}$
to $H_{2}$. Then, there exists a type-II graph homomorphism
$f_{3}$ from $G_{1}\times H_{1}$ to $G_{2}\times H_{2}$.}
\end{theorem}
\begin{proof}{ Let $f_{3}:V(G_{1}\times H_{1})\rightarrow
V(G_{2}\times H_{2})$ be defined the function which assigns to
each $(g,h)\in V(G_{1}\times H_{1})$,
$f_{3}((g,h))=(f_{1}(g),f_{2}(h))$. For each $\{\
(x_{1},y_{1}),(x_{2},y_{2})\ \}\in E(G_{1}\times H_{1})$,
$\{x_{1},x_{2}\}\in E(G_{1})$ and $\{y_{1},y_{2}\}\in E(H_{1})$.
Therefore, $\{\ f_{1}(x_{1}),f_{1}(x_{2})\ \}\in E(G_{2})$ and
$\{\ f_{2}(y_{1}),f_{2}(y_{2})\ \}\in E(H_{2})$. Hence, $\{\
f_{3}((x_{1},y_{1})),f_{3}((x_{2},y_{2}))\ \}\in E(G_{2}\times
H_{2})$ and therefore, the property 1 of type-II graph
homomorphisms holds. Now for each $\{\ (a,b),(c,d)\ \}\in
E(G_{2}\times H_{2})$ and each $(\alpha,\beta)\in f_{3}^{-1}(\
(c,d)\ )$, $\alpha\in f_{1}^{-1}(c)$ and $\beta\in f_{2}^{-1}(d)$.
So, there exist $x\in f_{1}^{-1}(a)$ and $y\in f_{2}^{-1}(b)$ such
that $\{x,\alpha\}\in E(G_{1})$ and $\{y,\beta\}\in E(H_{1})$,
hence, $(x,y)\in f_{3}^{-1}(\ (a,b)\ )$ and $\{\
(x,y),(\alpha,\beta)\ \}\in E(G_{1}\times H_{1})$. So, the
property 2 of type-II graph homomorphisms holds, too.
Consequently, $f_{3}$ is a type-II graph homomorphism.

}\end{proof}

We know that if $f$ is a type-II graph homomorphism from $G$ to
$H$ and $k\in {\rm Fall}(H)$, then, $k\in {\rm Fall}(G)$. Also,
for each graph $G$ and each natural number $k$, $k\in {\rm
Fall}(G)$ iff there exists a type-II graph homomorphism from $G$
to $K_{k}$. Therefore, the previous theorem implies the following
corollary.

\begin{cor}{ Let $n\in \mathbb{N}$ and for each $i\in [n]$,
$G_{i}$ be a graph and $k_{i}\in {\rm Fall}(G_{i})$. Then, there
exists a type-II graph homomorphism from $\times_{i=1}^{n}G_{i}$
to $\times_{i=1}^{n}K_{k_{i}}$ and ${\rm
Fall}(\times_{i=1}^{n}K_{k_{i}})\subseteq {\rm
Fall}(\times_{i=1}^{n}G_{i})$. Also,
$\chi_{f}(\times_{i=1}^{n}G_{i})\leq\chi_{f}(\times_{i=1}^{n}K_{k_{i}})\leq\psi_{f}(\times_{i=1}^{n}K_{k_{i}})\leq\psi_{f}(\times_{i=1}^{n}G_{i})$.
These inequalities can easily extend to more inequalities in
general. For example, in the case $n=2$, $\chi_{f}(G_{1}\times
G_{2})\leq\left\{\begin{array}{l} \chi_{f}(K_{k_{1}}\times
G_{2})\\ \chi_{f}(G_1\times K_{k_{2}})
\end{array}\right.\leq\chi_{f}(K_{k_{1}}\times
K_{k_{2}})\leq\psi_{f}(K_{k_{1}}\times
K_{k_{2}})\leq\left\{\begin{array}{l}\psi_{f}(K_{k_{1}}\times
G_{2})\\
\psi_{f}(G_{1}\times
K_{k_{2}})\end{array}\right.\leq\psi_{f}(G_{1}\times G_{2})$.

}
\end{cor}

Dunbar, et al. in \cite{dun} showed that for each $m,n\in
\mathbb{N}\setminus\{1\}$, ${\rm Fall}(K_{m}\times
K_{n})=\{m,n\}$. They also showed that if $n\in
\mathbb{N}\setminus\{1\}$ and for each $i\in [n]$, $r_{i}\in
\mathbb{N}\setminus\{1\}$, then, $\{r_{1},...,r_{n}\}\subseteq
{\rm Fall}(\times_{i=1}^{n}K_{r_{i}})$. They constructed a fall
6-coloring of $K_{2}\times K_{3}\times K_{4}$ and asked for
conditions for a finite and with more than two elements set
$S:=\{r_{1},...,r_{n}\}\subseteq\mathbb{N}\setminus\{1\}$ for
which $S\subsetneqq {\rm Fall}(\times_{i=1}^{n}K_{r_{i}})$.

\begin{theorem}{ Let $n\geq3$, $S:=\{r_{1},...,r_{n}\}\subseteq \mathbb{N}\setminus
\{1\}$, $r_{1}<r_{2}<...<r_{n}$ and $S$ contain at least one even
integer. Then, $S\subsetneqq {\rm
Fall}(\times_{i=1}^{n}K_{r_{i}})$, besides, ${\rm
Fall}(\times_{i=1}^{n}K_{r_{i}})$ contains an integer greater
than $r_{n}$.

}\end{theorem}

\begin{proof}{

There are five cases.

Case I) The case that $r_{1}=2$. In this case, let $t\in
\{r_{1},...,r_{n}\}\setminus \{r_{1},r_{n}\}$. Consider
$K_{2}\times K_{t}\times K_{r_{n}}$. Let $\sigma$ be a
disarrangement of $[t]$ ( a permutation $\sigma$ of $[t]$ such
that for each $i\in [t]$, $\sigma (i)\neq i$). Obviously, $\{\
\{(1,i,1),(1,\sigma(i),2),(2,i,2),(2,\sigma(i),1)\}\  |\ 1\leq
i\leq t\ \}\ \bigcup\ \{\ \{(x,y,z)\ |\ (x,y,z)\in K_{2}\times
K_{t}\times K_{r_{n}},\ z=i\ \}\ |\ 3\leq i\leq r_{n}\ \}$ is the
set of color classes of a fall $(r_{n}+t-2)$-coloring of
$K_{2}\times K_{t}\times K_{r_{n}}$. But $r_{n}+t-2>r_{n}$ and
therefore, in this case, ${\rm Fall}(K_{2}\times K_{t}\times
K_{r_{n}})$ contains an integer greater than $r_{n}$.

Case II) The case that $2<r_{1}$ and $\{r_{1},...,r_{n}\}$
contains at least two distinct even integers such that one of
them is $r_{n}$ and the other is $r_{s}$ that $s\in
\{1,...,n-1\}$. Let $r_{j}\in \{r_{1},...,r_{n}\}\setminus
\{r_{s},r_{n}\}$. Consider $K_{r_{s}}\times K_{r_{j}}\times
K_{r_{n}}$ and a disarrangement $\sigma$ of $[r_{j}]$. For each
$1\leq t\leq r_{j}$, color the vertices $(1,t,1),(1,\sigma
(t),2),(2,t,2)$ and $(2,\sigma (t),1)$ with the color $t$ and
color each other vertex $(x,y,z)$ with the color
$\lfloor\frac{x-1}{2}\rfloor(\frac{r_{j}r_{n}}{2})+\lfloor\frac{z-1}{2}\rfloor
r_{j}+{\rm \ the\ color\ of}\
(x-2\lfloor\frac{x-1}{2}\rfloor,y,z-2\lfloor\frac{z-1}{2}\rfloor)$.
This is a fall $\frac{r_{s}r_{j}r_{n}}{4}$-coloring of
$K_{r_{s}}\times K_{r_{j}}\times K_{r_{n}}$. Since $2<r_{1}$,
$\frac{r_{s}r_{j}r_{n}}{4}>\max\{r_{s},r_{j},r_{n}\}$. Hence,
Theorem \ref{thm4} implies that ${\rm
Fall}(\times_{i=1}^{n}K_{r_{i}})$ contains an integer greater
than $r_{n}$.

Case III) The case that $2<r_{1}$ and $\{r_{1},...,r_{n}\}$
contains at least two distinct even integers such that none of
them is $r_{n}$. Similar to the case II, ${\rm
Fall}(\times_{i=1}^{n}K_{r_{i}})$ contains an integer greater
than $r_{n}$.

Case IV) The case that $2<r_{1}$ and $\{r_{1},...,r_{n}\}$
contains exactly one even integer and $r_{n}$ is even. In this
case, consider $K_{r_{n-2}-1}\times K_{r_{n-1}}\times K_{r_{n}}$
and a disarrangement $\sigma$ of $[r_{n-1}]$. For each $1\leq
t\leq r_{n-1}$, color the vertices $(1,t,1),(1,\sigma
(t),2),(2,t,2)$ and $(2,\sigma (t),1)$ with the color $t$ and
color each other vertex $(x,y,z)$ of $K_{r_{n-2}-1}\times
K_{r_{n-1}}\times K_{r_{n}}$ with the color
$\lfloor\frac{x-1}{2}\rfloor(\frac{r_{n-1}r_{n}}{2})+\lfloor\frac{z-1}{2}\rfloor
r_{n-1}+{\rm \ the\ color\ of}\
(x-2\lfloor\frac{x-1}{2}\rfloor,y,z-2\lfloor\frac{z-1}{2}\rfloor)$.
Also, color each vertex $(r_{n-2},y,z)$ of $K_{r_{n-2}}\times
K_{r_{n-1}}\times K_{r_{n}}$ with the color
$\frac{(r_{n-2}-1)r_{n-1}r_{n}}{4}+1$. Therefore, a fall
$(\frac{(r_{n-2}-1)r_{n-1}r_{n}}{4}+1)$-coloring of
$K_{r_{n-2}}\times K_{r_{n-1}}\times K_{r_{n}}$ and also of
$\times_{i=1}^{n} K_{r_{i}}$ yields. But,
$\frac{(r_{n-2}-1)r_{n-1}r_{n}}{4}+1>r_{n}$. Thus,  ${\rm
Fall}(K_{2}\times K_{t}\times K_{r_{n}})$ and therefore ${\rm
Fall}(\times_{i=1}^{n}K_{r_{i}})$ contains an integer greater than
$r_{n}$.

Case V) The case that $2<r_{1}$ and $\{r_{1},...,r_{n}\}$
contains exactly one even integer and $r_{n}$ is odd. In this
case, similar to the case IV, ${\rm
Fall}(\times_{i=1}^{n}K_{r_{i}})$ contains an integer greater
than $r_{n}$.

Accordingly, in all cases, $\{r_{1},...,r_{n}\}\subsetneqq {\rm
Fall}(\times_{i=1}^{n}K_{r_{i}})$. Besides, ${\rm
Fall}(\times_{i=1}^{n}K_{r_{i}})$ contains an integer greater
than $r_{n}$.

} \end{proof}

Even though Dunbar, et al. in \cite{dun} constructed a fall
6-coloring of $K_{2}\times K_{3}\times K_{4}$, this theorem also
shows that in the corollary \ref{twoinequalities}, the inequality
$\max\{\ \psi_{f}(G_{i})\ |\ i\in [n]\ \}\leq
         \psi_{f}(\times_{i=1}^{n}G_{i})$
can be strict in many cases. But we conjecture that the inequality
$\chi_{f}(\times_{i=1}^{n}G_{i})\leq \min\{\ \chi_{f}(G_{i})\ |\
i\in
 [n]\ \}$ is always an equality.

\begin{con}{
For each $n\in \mathbb{N}$ and for each arbitrary graphs
$G_{1},\ldots,G_{n}$ such that for each $i\in [n]$, ${\rm
Fall}(G_{i})\neq\emptyset$, the following equality holds.

$\chi_{f}(\times_{i=1}^{n}G_{i})=\min\{\chi_{f}(G_{i})|\ i\in
 [n]\}$.
} \end{con}

\section{Fall colorings of union of graphs}

Let $n\in \mathbb{N}$ and for each $1\leq i\leq n$, $G_{i}$ be a
graph. The graph $(\ \bigcup_{i=1}^{n}\ (\{i\}\times V(G_{i}))\
,\ \bigcup_{i=1}^{n}\ \{\ \{(i,x),(i,y)\}\ |\ \{x,y\}\in
E(G_{i})\ \}\ )$ is called the union graph of $G_{1},...,G_{n}$
and is denoted by $\biguplus_{i=1}^{n}G_{i}$.

The following obvious theorem describes fall colorings of union
of graphs.

\begin{theorem}{
Let $n\in \mathbb{N}$ and for each $1\leq i\leq n$, $G_{i}$ be a
graph. Then, the following three statements hold.

1) If ${\rm Fall}(\biguplus_{i=1}^{n}G_{i})\neq\emptyset$, then,
for each $1\leq i\leq n$, ${\rm Fall}(G_{i})\neq\emptyset$.

2) ${\rm Fall}(\biguplus_{i=1}^{n}G_{i})=\bigcap_{i=1}^{n}{\rm
Fall}(G_{i})$.

3) If ${\rm Fall}(\biguplus_{i=1}^{n}G_{i})\neq\emptyset$, then, $
\chi_{f}(\biguplus_{i=1}^{n}G_{i})$=$\min\bigcap_{i=1}^{n}{\rm
Fall}(G_{i})$ and
$\psi_{f}(\biguplus_{i=1}^{n}G_{i})=\max\bigcap_{i=1}^{n}{\rm
Fall}(G_{i})$.}
\end{theorem}

Since any graph $G$ is isomorphic to any union graph of all its
connected components, the following corollary yields immediately.

\begin{cor}{ Let $G$ be a graph and $G_{i}$ $(1\leq i\leq n)$
be all its connected components. Then, the following three
statements hold.

1) If ${\rm Fall}(G)\neq\emptyset$, then, for each $1\leq i\leq
n$, ${\rm Fall}(G_{i})\neq\emptyset$.

2) ${\rm Fall}(G)=\bigcap_{i=1}^{n}{\rm Fall}(G_{i})$.

3) If ${\rm Fall}(G)\neq\emptyset$, then,
$\chi_{f}(G)=\min\bigcap_{i=1}^{n}{\rm Fall}(G_{i})$ and
$\psi_{f}(G)=\max\bigcap_{i=1}^{n}{\rm Fall}(G_{i})$.

}
\end{cor}

\section{Restriction of fall $t$-colorings of a graph into proper $t$-colorings of graphs in a specified set}

Now we prove that fall $k$-colorings of a graph can be reduced
into proper $k$-colorings of graphs in a specified set.

Let $G$ be a graph and $1\leq t\leq \delta(G)+1$ be a fixed
natural number. For each $v\in V(G)$, choose $t-1$ arbitrary
elements of $N_{G}(v)$ and join these $t-1$ vertices to each
other and name the new graph $H$. Let $\widehat{G_{t}}$ be the set
of all graphs $H$ constructed like this.

\begin{theorem}{ For each $1\leq t\leq \delta(G)+1$, $t\in {\rm Fall}(G)$  iff $t\in \{\chi(H)|\ H\in
\widehat{G_{t}}\}$. Specially, ${\rm Fall}(G)=\bigcup_{i=1}^{
\delta (G)+1}(\ \{\chi(H)|\ H\in \widehat{G_{i}}\}\bigcap \{i\}\
)$.

}\end{theorem}

\begin{proof}{

Let $1\leq t\leq \delta(G)+1$. If $t\in \{\chi(H)|\ H\in
\widehat{G_{t}}\}$, then, there exists a graph $H$ in
$\widehat{G_{t}}$ such that $\chi(H)=t$ and there exists a
$t$-coloring $f$ of $H$. This coloring $f$ of $H$, is obviously a
fall $t$-coloring of $G$ and therefore, $t\in {\rm Fall}(G)$.
Conversely, if $t\in {\rm Fall}(G)$, then, there exists a fall
$t$-coloring $g$ of $G$. For each $v\in V(G)$, there exist $t-1$
elements of $N_{G}(v)$ such that the set of their colors and the
color of $v$ is equal to $[t]$, join all of them to each other to
construct a new graph $T$ in $\widehat{G_{t}}$. The fall
$t$-coloring $g$ of $G$ is obviously a $t$-coloring of $T$, also
$\omega(T)\geq t$, thus, $\chi(T)=t$ and $t\in \{\chi(H)|\ H\in
\widehat{G_{t}}\}$. The second part of the theorem follows
immediately.

}\end{proof}

Restricting this theorem into $r$-regular graphs and $t=r+1$,
yields a beautiful proposition of \cite{mah} but in different
terminologies.

\begin{pro}{
For each $r$-regular graph $G$, $r+1\in {\rm Fall}(G)$ iff
$\chi(G^{(2)})=r+1$, where $G^{(2)}=(\ V(G)\ ,\ \{\ \{x,y\}\ |\
x,y\in V(G),\ x\neq y,\ d_{G}(x,y)\leq2\ \}\ )$.}
\end{pro}

\section{Fall Colorings of Mycielskian of graphs}

Let $G:=(\ \{x_{1},\ldots,x_{n}\}\ ,\ E(G)\ )$ be a graph. The
Mycielskian of $G$ ( denoted by $M(G)$ ) is a graph with $2n+1$
vertices $x_{1},\ldots,x_{n},y_{1},\ldots,y_{n},z$ with edge set
$E(G)\ \bigcup\  \{\ \{y_{i},x_{j}\}\ |\ i,j\in [n],\
\{x_{i},x_{j}\}\in E(G)\ \}\ \bigcup\  \{\ \{z,y_{i}\}\ |\ i\in
[n]\ \}$.

For example, $M(K_{2})$ is $C_{5}$. We know that ${\rm
Fall}(M(K_{2}))={\rm Fall}(C_{5})=\emptyset$. Now we prove that
for each graph $G$, ${\rm Fall}(M(G))=\emptyset$.

\begin{theorem}{ For each graph $G$, ${\rm Fall}(M(G))=\emptyset$.}
\end{theorem}
\begin{proof}{ If $E(G)=\emptyset$, then, $M(G)$ has at least one
isolated vertex and also at least one edge. Therefore, ${\rm
Fall}(M(G))=\emptyset$. Now we prove the theorem for the case
$E(G)\neq\emptyset$. If $E(G)\neq\emptyset$ and ${\rm
Fall}(M(G))\neq\emptyset$, then, there exists a fall $k$-coloring
$f$ of $M(G)$ for some $k\in \mathbb{N}$. Since
$E(G)\neq\emptyset$, there exists an integer $i_{0}\in [n]$ such
that $f(x_{i_{0}})\neq f(z)$ and since for each $j\in [n]$,
$f(y_{j})\neq f(z)$ and $f$ is a fall $k$-coloring, there exists
$i_{1}\in [n]$ such that $x_{i_{1}}\in N_{G}(x_{i_{0}})$ and
$f(x_{i_{1}})=f(z)$. Since for each $i\in [n]$ with $f(x_{i})\neq
f(z)$, $N(y_{i})\setminus\{z\}\subseteq N(x_{i})$, so
$f(x_{i})\in \{f(y_{i}),f(z)\}$, on the other hand, $f(x_{i})\neq
f(z)$, hence, $f(x_{i})=f(y_{i})$. This immediately shows that
each color of $[k]$ appears on the neighborhood of $y_{i_{1}}$,
which is a contradiction. Hence, ${\rm Fall}(M(G))=\emptyset$.

}
\end{proof}

\section{Fall colorings of complement of bipartite graphs}

Complements of bipartite graphs are very interesting graphs,
because in each proper $k$-coloring, the cardinality of each color
class is at most 2. The following theorem characterizes all fall
colorings of this type of graphs.

\begin{theorem}{
Let $G$ be a bipartite graph. Then, ${\rm Fall}(G^{c})\subseteq \
\{\ \chi (G^{c})\ \}$. Besides, it is polynomial to decide
whether or not ${\rm Fall}(G^{c})=\{\ \chi (G^{c})\ \}$ .

}\end{theorem}

\begin{proof}{
If ${\rm Fall}(G^{c})\neq\emptyset$, then, $\exists k\in {\rm
Fall}(G^{c})$. Suppose that $f$ is a fall $k$-coloring of $G^{c}$.
Obviously, each color class of $f$ is either of the form $\{x\}$
or of the form $\{y,z\}$ such that $y\in A$ and $z\in B$. A color
class is of the form $\{x\}$ iff $x$ is an isolated vertex of the
graph $G$. Therefore, the set of color classes of $f$ is the
union of $\{\ \{x\}\ |\ x\in V(G),\ deg_{G}(x)=0\ \}$ and the set
of edges of a perfect matching of the induced subgraph of $G$ on
$\{\ x\ |\ x\in V(G),\ deg_{G}(x)>0\ \}$, also,
$k=|V(G)|-\frac{1}{2}|\{\ x\ |\ x\in V(G),\ deg_{G}(x)>0\ \}|$.
Therefore, ${\rm Fall}(G^{c})\subseteq \{\ |V(G)|-\frac{1}{2}|\{\
x\ |\ x\in V(G),\ deg_{G}(x)>0\ \}|\ \}$. Besides, if ${\rm
Fall}(G^{c})\neq\emptyset$, then, the induced subgraph of $G$ on
$\{\ x\ |\ x\in V(G),\ deg_{G}(x)>0\ \}$ has a perfect matching,
in this case, obviously, $\chi (G^{c})=|V(G)|-\frac{1}{2}|\{\ x\
|\ x\in V(G),\ deg_{G}(x)>0\ \}|$, and consequently, ${\rm
Fall}(G^{c})=\{\  \chi (G^{c})\ \}$. Therefore, for each bipartite
graph $G$,  ${\rm Fall}(G^{c})\subseteq \{\ \chi (G^{c})\ \}$. We
know that if ${\rm Fall}(G^{c})\neq\emptyset$, then, the induced
subgraph of $G$ on $\{\ x\ |\ x\in V(G),\ deg_{G}(x)>0\ \}$ has a
perfect matching. Conversely, if the induced subgraph of $G$ on
$\{\ x\ |\ x\in V(G),\ deg_{G}(x)>0\ \}$ has a perfect matching,
then, the union of $\{\ \{x\}\ |\ x\in V(G),\ deg_{G}(x)=0\ \}$
and the edge set of each perfect matching of the induced subgraph
of $G$ on $\{\ x\ |\ x\in V(G),\ deg_{G}(x)>0\ \}$ is the set of
color classes of a fall $(\ |V(G)|-\frac{1}{2}|\{\ x\ |\ x\in
V(G),\ deg_{G}(x)>0\ \}|\ )$-coloring of $G^{c}$ and therefore,
${\rm Fall}(G^{c})\neq\emptyset$. Accordingly, ${\rm
Fall}(G^{c})=\{\  \chi (G^{c})\ \}$ iff ${\rm
Fall}(G^{c})\neq\emptyset$ iff the induced subgraph of $G$ on
$\{\ x\ |\ x\in V(G),\ deg_{G}(x)>0\ \}$ has a perfect matching.
Since the problem of deciding whether or not the induced subgraph
of $G$ on $\{\ x\ |\ x\in V(G),\ deg_{G}(x)>0\ \}$ has a perfect
matching, is a polynomial time  problem, thus, it is polynomial
time to decide whether or not ${\rm Fall}(G^{c})=\{\ \chi (G^{c})\
\}$. }\end{proof}

\textbf{Acknowledgements}
\\
\\
The author wishes to thank Ali Jamalian who introduced him the
subject of fall coloring and the reference \cite{dun}. Also, he
would like to thank Hossein Hajiabolhassan, Meysam Alishahi and
Ali Taherkhani for their useful comments.


\begin{thebibliography}{4}

\bibitem{MR2096633}
W. Dong and B. G. Xu.
\newblock Fall colorings of {C}artesian product graphs and regular graphs.
\newblock {\em J. Nanjing Norm. Univ. Nat. Sci. Ed.}, 27(3):17--21, 2004.

\bibitem{MR2193924}
N. Drake, R. Laskar, J. Lyle,
\newblock Independent domatic partitioning or fall coloring of strongly chordal
  graphs.
\newblock {\em Congr. Numer.}, 172:149--159, 2005.
\newblock 36th Southeastern International Conference on Combinatorics, Graph
  Theory, and Computing.

\bibitem{dun} J. E. Dunbar, S. M. Hedetniemi, S. T. Hedetniemi, D. P. Jacobs, J. Knisely, R. C. Laskar, D. F. Rall, Fall colorings of graphs, {\em J. Combin. Math. Combin. Comput}\ {\bf
33}(2000), 257-273. papers in honor of Ernest J. Cockayne.

\bibitem{mah} M. Ghebleh, L. A. Goddyn, E. S. Mahmoodian, M. Verdian-Rizi, Silver cubes, {\em Graphs Combin.}\ {\bf
24}(2008), 429-442.

\bibitem{las} R. Laskar, J. Lyle, Fall coloring of bipartite graphs and cartesian products of graphs, {\em Discrete Applied Mathematics}\ {\bf
157}(2009), 330-338.
\end{thebibliography}
\end{document}